\documentclass{amsart} \usepackage{amssymb,amsfonts,amscd}
\newtheorem{theorem}{Theorem}[section]
\newtheorem{lemma}[theorem]{Lemma}
\newtheorem{corollary}[theorem]{Corollary}
\newtheorem{proposition}[theorem]{Proposition}
\theoremstyle{remark}  

\theoremstyle{definition}

\numberwithin{equation}{section} \makeatother

\DeclareMathOperator{\Cdb}{{\mathbb C}}

\DeclareMathOperator{\Ndb}{{\mathbb N}}

\begin{document}

\title{Metric characterizations II} 
\author[D. P. Blecher]{David P. Blecher}
\author[Matthew Neal]{Matthew Neal}

\address{Department of Mathematics, University of Houston, Houston, TX
77204-3008}
 \email[David P.
Blecher]{dblecher@math.uh.edu}
\address{Department of Mathematics,
Denison University, Granville, OH 43023}
\email{nealm@denison.edu} \thanks{Presented at 
AMS/SAMS Satellite Conference on Abstract Analysis, University of Pretoria, South Africa, 5-7 December 2011.
Revision of 2/23/2012 (Example after theorem 3.2 added).} \begin{abstract} The present paper is a sequel to our paper
``Metric characterization of isometries and of unital operator spaces and systems''.  We characterize certain common objects in 
the theory of operator spaces
(unitaries, unital operator spaces, operator systems, operator
algebras, and so on), in terms which are
purely {\em linear-metric}, by which we mean that they
only use the vector space structure of the space and its
matrix norms.  In the last part we give some
characterizations of operator algebras (which are not
linear-metric in our strict sense described in the paper).
\end{abstract}

\maketitle

\let\text=\mbox

\section{Introduction}

The present paper is a sequel to our paper
``Metric characterization of isometries and of unital operator spaces and systems''
 \cite{BN}.   The goal of both papers is to 
characterize certain common objects  in
the theory of operator spaces
(unitaries, 
unital operator spaces, operator systems, operator 
algebras, and so on), in terms which are 
purely {\em linear-metric}, by which we mean that they
only use the vector space structure of the space and its
matrix norms, in the spirit of Ruan's matrix norm characterization
of operator spaces \cite{Ru}, not mentioning
products, involutions, or any kind of function such as
 linear maps on the space.
 In the present paper we give new linear-metric characterizations
of unital operator spaces (or equivalently, of `unitaries' in 
an operator space).   Some of our characterizations should be useful
in future.  Others may look cumbersome, but their virtue is
 that it is nice to know that such characterizations 
exist (only using the norm and/or vector space structure).      
An example of one of our new characterizations is the following:
 
\begin{theorem}  \label{intro}  If $X$ is
an operator space and $u \in {\rm Ball}(X)$ then $(X,u)$
is a unital operator space (or equivalently, 
$u$ is a unitary in $X$) iff for all $n \in \Ndb, x \in M_n(X)$, 
$$\max \{ \Vert u_n + i^k x \Vert : k = 0,1,2,3 \} \geq \sqrt{1 + \Vert x \Vert_n}
.$$
\end{theorem}

Here $u_n$ is the diagonal matrix $u \otimes I_n$
in $M_n(X)$ with $u$ in
each diagonal entry.   One advantage of this characterization is that
it avoids `matrices of matrices', in contrast to our other ones.   We will give a convincing  
illustration of the use of this criterion after its proof (after Theorem \ref{tc3} below). 

We also give several other 
assorted results and observations, 
most of these being complements to various results in \cite{BN}.
   The structure of our paper is as follows: In Section 2 we
present some matrix norm formulae that will be used later in
the paper.  In Section 3 (resp.\ Section 4)
we give new linear-metric characterizations
of unital operator spaces (resp.\ operator systems).  In
Section 3 we also relate unital operator spaces to our previous paper
\cite{BNc}, by characterizing compact projections
in a $C^*$-algebra in terms of unital operator spaces.
In the remaining sections,  
 we characterize operator algebra structures on operator 
spaces in various ways.   For example we give new variants of the 
characterization of 
operator algebras due to the first author with Ruan and Sinclair \cite{BRS}.    
In various remarks scattered through our paper we indicate where a result
may be strengthened, or give counterexamples ruling out certain directions 
of enquiry.

Turning to definitions, all vector spaces are over the complex
field $\Cdb$.  The letters $H, K$ are usually reserved for Hilbert
spaces.  A given cone in a space $X$ will sometimes be written as
$X_+$, and
 $X_{\rm sa} = \{ x \in X : x = x^* \}$ assuming that there
is an involution $*$ around.  The reader may consult \cite{BLM},
or one of the other books on operator spaces, for more information
if needed below.   All normed (or operator) spaces are
assumed to be complete.  
A {\em unital operator space} is a subspace
 of a unital $C^*$-algebra containing the identity. More abstractly, a
unital operator space is a pair $(X,u)$ consisting of an operator
space $X$ containing a fixed element $u$ such that there exists a
Hilbert space $H$ and a complete isometry $T : X \to B(H)$ with
$T(u) = I_H$.   In this case we also say that 
$u$ is a {\em unitary in} $X$.
An {\em operator system} is a selfadjoint subspace of a
unital $C^*$-algebra containing the identity; or 
more abstractly, a unital operator
space $(X,u)$ for which there
exists a linear complete isometry $T : X \to B(H)$ with $T(u) = I_H$
and $T(X)$ selfadjoint.  An {\em operator algebra} is
an operator space $A$ which is an algebra such that there
exists a completely isometric homomorphism from $A$ 
into a $C^*$-algebra.  An operator algebra is 
{\em unital} if it has an identity of norm $1$.
  
A {\em TRO} (ternary ring of
operators) is a closed
subspace $Z$ of a C*-algebra, or of $B(K,H)$,  such that $Z Z^* Z \subset Z$.
We refer to e.g.\ \cite{Ham,BLM} for the basic theory of TROs.
A {\em ternary morphism} on a TRO $Z$ is a linear map $T$ such that
$T(x y^* z) = T(x) T(y)^* T(z)$ for all $x, y, z \in Z$.
We write $Z Z^*$ for the closure of the linear span of
products $z w^*$ with $z, w \in Z$, and similarly for $Z^* Z$. These are
$C^*$-algebras.
 The {\em ternary envelope}
of an operator space  $X$ is a pair $({\mathcal T}(X),j)$ consisting of
a TRO ${\mathcal T}(X)$ and a  completely isometric linear map $j : X
\to {\mathcal T}(X)$, such that
${\mathcal T}(X)$ is generated by $j(X)$ as a TRO (that is, there is no closed
subTRO containing $j(X)$), and which has the following
property: given any completely isometric
linear map $i$ from $X$ into a TRO $Z$ which is
generated by $i(X)$, there exists a (necessarily unique and surjective)
ternary morphism $\theta : Z \to {\mathcal T}(X)$ such that
$\theta \circ i = j$.
If $(X,u)$ is a unital operator space then the ternary envelope
may be taken to be the {\em $C^*$-envelope} of e.g.\ \cite[Section 4.3]{BLM};
this is a $C^*$-algebra $C^*_e(X)$ with identity $u$.
If $X$ is an operator system then $X$ is  a selfadjoint unital subspace
of $C^*_e(X)$.

An element $u$ in an operator space 
$X$ is called a {\em coisometry} (resp.\ {\em isometry})
{\em in} $X$, if $X$ may be linearly completely isometrically 
embedded in a TRO $Z$ such that $u u^* = 1_{Z Z^*}$ 
(resp.\  $u^* u = 1_{Z^* Z}$).  In this case, $Z$ may be taken to be 
the ternary envelope of $X$, or it may be taken to be $B(K,H)$ and the 
$1$ in the last line replaced by the identity operator on the 
Hilbert space $H$ (resp.\  $K$).
Coisometries and isometries in $X$ were characterized purely linear-metrically
in \cite{BN} (see also Theorem \ref{uospr} below).  Also, $u$ is
a unitary in $X$ iff it is both a coisometry and an isometry in $X$
(see \cite[Lemma 2.3]{BN}).    
 
\section{Some matrix norm formulae}

We collect several formulae for matrix norms that we use later in the paper,
most of which are known. 

\begin{lemma} \label{abbal} Let $A$ be a $C^*$-algebra (or operator space). We have  
\begin{equation} \label{abba} \Vert \left[ \begin{array}{ccl} a & b \\
b & a \end{array} \right] \Vert = \max \{ ||a+b|| , ||a-b|| \} , \qquad 
a, \, b \in A . \end{equation}
\end{lemma} \begin{proof} This is well known: the map 
 taking the $2 \times 2$  matrix above to $(a + b,a-b)$, is a faithful $*$-homomorphism, from the
$C^*$-algebra of such matrices into $A \oplus^\infty A$.
\end{proof} 

\begin{lemma} \label{ambbal}  Let $A$ be a $C^*$-algebra (or operator space). We have  
 \begin{equation} \label{ambba} \Vert \left[ \begin{array}{ccl} a & -b \\ b & a \end{array}
 \right] \Vert = \max \{ ||a+ib|| , ||a-ib|| \} ,  \qquad
a, \,  b \in A . \end{equation}
\end{lemma} \begin{proof} 
To see this, apply (\ref{abba}) with $b$ replaced by $ib$, then
multiply, first, the second row by $-i$, and
second, the second column by $i$.
\end{proof}

Let $X$ be an operator space, and  $v \in X$.
If $n \in \Ndb$ and $x  \in M_n(X)$ we write  $$t^v_x = \left[ \begin{array}{ccl} v_n & x \\ 0 & v_n \end{array} \right] .$$
If $(X,v)$ is a unital operator space, and we identify
$v = 1$, then we  write $t^v_x$ as $t_x$.

\begin{lemma} \label{1x01l}
 If $X$ is a unital operator space then
  \begin{equation} \label{1x01} \Vert t_x  \Vert^2 = \frac{1}{2}[2 + \Vert x \Vert^2  + \Vert x \Vert
 \sqrt{\Vert x \Vert^2 + 4}] \geq 1 + \Vert x \Vert,  \qquad n \in \Ndb, \; \, x  \in M_n(X). \end{equation} 
Thus $\Vert t_x \Vert
  \geq \sqrt{1 + \Vert x \Vert}$.
\end{lemma} \begin{proof}   This is e.g.\  a consequence of the 
more general formula (2.1) in \cite{BM2}.
\end{proof}
 
We will see in Theorem \ref{tcl} 
that (the matricial version of) this condition characterizes unital operator
 spaces.

Let $X$ be an operator space  possessing a 
conjugate linear involution $* : X \to X$, and suppose that $v \in X$.  
If $n \in \Ndb$ and $x  = [x_{ij}]  \in M_n(X)$ 
define
$x^* = [x^*_{ji}]$, and   $$s^v_x = \left[ \begin{array}{ccl} v_n & x \\ x^* & v_n \end{array}
 \right] \; , \; \; \, \; \; r^v_x = \left[
 \begin{array}{ccl}    v_n & x \\ -x^* & v_n \end{array} \right] .$$
If $(X,v)$ is a unital operator space, and we identify  
$v = 1$, then we write $s^v_x$ as $s_x$ and $r^v_x$ as $r_x$.

\begin{lemma} \label{sf}  If $X$ is an operator system 
then \begin{equation} \label{sx} \Vert s_x \Vert =  1 +  ||x||
, \qquad n \in \Ndb, \; \, x  \in M_n(X). \end{equation}
and 
\begin{equation} \label{rx} \Vert r_x \Vert = \sqrt{1 + \Vert x \Vert^2}
, \qquad n \in \Ndb, \; \, x  \in M_n(X). \end{equation} \end{lemma}

\begin{proof}  The first is
  well known (and an easy exercise).  The second is from \cite{BN}, and is
a  simple application of the $C^*$-identity. 
\end{proof}

In Proposition \ref{chos} we show that formula (\ref{rx}) characterizes operator systems.

\section{New metric-linear characterizations of unital operator spaces}

We now show that
the inequality in Lemma \ref{1x01l} characterizes unital operator spaces.  Before we prove
  this, note that $\Vert t^v_x \Vert
  \geq \sqrt{1 + \Vert x \Vert_n}$ for all $x \in X$ iff
  $$\left\| \left[ \begin{array}{ccl} \lambda v_n & x \\
0 & \lambda  v_n \end{array}
 \right] \right\| \geq \sqrt{|\lambda|^2 + |\lambda|  \Vert x
\Vert_n}$$ for all $\lambda \in \Cdb$ and $x \in M_n(X)$. In fact it
is enough to take $\lambda > 0$ here, and $\Vert x \Vert = 1$.

\begin{theorem} \label{tcl}  If $v \in {\rm Ball}(X)$ then $(X,v)$
is a unital operator space iff $\Vert t^v_x \Vert
  \geq \sqrt{1 + \Vert x \Vert_n}$ for all $x \in M_n(X)$ of small norm,
and all $n \in \Ndb$.
\end{theorem}
\begin{proof}  The one direction is Lemma \ref{1x01l}.
If the norm condition holds then for all $x \in {\rm Ball}(M_n(X))$
(of small norm), $$1 + \Vert
x \Vert_n \leq \Vert t_x \Vert^2 = \Vert t_x^* t_x \Vert \leq 1 + \Vert x \Vert^2
+ \Vert v^* x \Vert,$$
where we are writing $v_n$ as $v$ for brevity.
Write $x = cy$ where $c > 0$ and $\Vert y \Vert = 1$, then
$c \leq c^2 + c \Vert v^* y \Vert$, so that $1 \leq c + \Vert v^* y \Vert$.
Hence $1  \leq \Vert v^* y \Vert$ if $\Vert y \Vert = 1$ (letting $c \searrow 0$).
This implies that $\Vert v^* x \Vert = \Vert x \Vert$ for all $x \in M_n(X)$.
Similarly, by using
$t_x t_x^*$ in the calculation above, we have $\Vert x v^* \Vert = \Vert x \Vert$.
Now by the proof of Theorem 2.4 in \cite{BN} we see that $v$ is a unitary in $X$.
\end{proof}

{\bf Remark.}  It is not enough that $\Vert t^v_x \Vert
  \geq \sqrt{2}$ if $\Vert x \Vert_n = 1$; this does not
characterize unital operator spaces.  To see this take $X = H^c$,
Hilbert column space.

   \begin{theorem} \label{tc3}   If $X$ is
an operator space and $v \in {\rm Ball}(X)$ then $(X,v)$
is a unital operator space iff $\max \{ \Vert v_n + i^k x \Vert : k = 0,1,2,3 \}
 \geq \sqrt{1 + \Vert x \Vert_n}$ for all $x \in M_n(X)$ of small norm,
and all $n \in \Ndb$.
\end{theorem}
\begin{proof}  ($\Leftarrow$) \  Apply Theorem \ref{tcl}, replacing
$x$ two lines above by the $2 \times 2$ matrix with $x$ in the $1$-$2$ corner and
zeroes elsewhere, and $v_n$ there by $v_{2n}$, noting that
$\Vert t_{i^k x} \Vert = \Vert t_x \Vert$.

($\Rightarrow$) \ By (\ref{abba}) and (\ref{1x01}), we have
$$\left\| \left[ \begin{array}{ccccl} 1 & x & 1 & x \\
x & 1 & -x & 1 \\
1 & x & 1 & x \\
-x & 1  &  x & 1 \end{array}
 \right] \right\| = \max \{ \Vert \left[ \begin{array}{ccl} 2 & 2 x \\
0 & 2 \end{array}
 \right] \Vert , \Vert \left[ \begin{array}{ccl} 0 & 0 \\ 2x & 0 \end{array}
 \right] \Vert \} \geq 2 \sqrt{1 + \Vert x \Vert_n} .$$
However the norm of the big matrix here is also $$\leq
 \Vert \left[ \begin{array}{ccl}
1 & x \\ x & 1 \end{array}
 \right] \Vert + \Vert \left[ \begin{array}{ccl} 1 & x \\ -x & 1 \end{array}
 \right] \Vert \leq 2 \max \{  \Vert v_n + i^k x \Vert : k = 0,1,2,3 \},$$
 by (\ref{abba}) and (\ref{ambba}).  \end{proof}

{\bf Remark.}  One may replace $i^k$ in the theorem by the set of
unimodular complex scalars (the proof is unchanged).  There is also an equivalent
rewriting of the `max' condition in the last theorem
 in terms of the cone ${\mathfrak F}_X$ highlighted in \cite{BRead}.
If $z \in {\mathfrak F}_{M_n(X)}$ let $c_k(z) = v_n + i^k(v_n-z) \in {\mathfrak F}_{M_n(X)}, k = 0,1,2,3$.
The condition then becomes: $\max \{ \Vert c_k(z)  \Vert : k = 0,1,2,3 \} \geq \sqrt{1 + \Vert v_n-z \Vert}$,
for all $z \in {\mathfrak F}_{M_n(X)}$.

\medskip

We give some illustrations of the use of this criterion in practice. 

\medskip

{\bf Examples 1).}    To see immediately that $c_0$ is not a unital operator space: For any norm $1$ element $\vec x \in c_0$, if $|x_n|$ is small enough then clearly $||\vec x + i^k  \vec e_n ||$ does
not dominate $\sqrt{2}$ for any $k$  (here $(\vec e_n)$ is the standard  basis).  So by  Theorem \ref{intro},  
 $c_0$  is not a  unital operator space (with any operator space structure $\{ \Vert \cdot \Vert_n \}_{n \geq 2}$).

\smallskip

{\bf 2). } A convincing and more nontrivial example is $S^1_2$, namely $M_2$ with the trace norm.  This example is also interesting 
because its `commutative variant' $\ell^1_2$  is well known to be a unital operator space, as  indeed also is $\ell^1$
and more generally  various `Fourier algebras' $B(G)$ and their noncommutative variants (see \cite[Section 3]{BM} and its methods).
Suppose $a \in M_2$ with trace$(|a|) = 1$.    Then there exist unitaries $u, v$ with $a = u d v$, where $d$ is
a diagonal matrix with positive entries $\alpha, \beta$ with $\alpha + \beta = 1$.  Let $x = t  u e_{21} v$, where $e_{21}$ is the usual matrix unit, and $t$ is a scalar.  Then 
$\Vert a + i^k x \Vert_1 = \Vert d + i^k t \,   e_{21} \Vert_1$.   Let $b = d +  i^k t \, e_{21}$, and by way of contradiction, assume
trace$(|b|) \geq \sqrt{1 + t}$ for small $t > 0$.  Let  $c =b^*b$, and 
suppose that the eigenvalues of $c$ are $r,s$.  Then trace$(c) = r + s$, and  trace$(|b|) = \sqrt{r}+ \sqrt{s}$.
Since trace$(|b|) \geq \sqrt{1 + t}$ we have $r + s + 2 \sqrt{rs}  \geq 1 + t$.  However, since $c$ has rows
$(\alpha^2 + t^2 , \overline{i^n} t \beta)$ and $(i^n t \beta, \beta^2)$, 
we see that trace$(c) = r + s = \alpha^2 +  \beta^2 + t^2$, and det$(c) = rs = (\alpha \beta)^2$.
Thus $$1 + t \leq \alpha^2 +  \beta^2 + t^2 + 2 \alpha \beta = (\alpha + \beta)^2 + t^2  = 1 + t^2,$$
a contradiction for small $t > 0$.  By Theorem \ref{tc3}, $a$ is not unitary, and  so $S^1_2$ is not a  unital operator space
(with any operator space structure $\{ \Vert \cdot \Vert_n \}_{n \geq 2}$).  This example also  illustrates one great advantage of this characterization over other ones: the criterion involves a linear combination of $u$ and $x$ rather than 
a matrix with these as entries.   Indeed we only needed $1 \times 1$ matrices in the computation above.

\smallskip

{\bf 3). }   By the same argument,  the 3 dimensional subspace $L^1_2$ of lower triangular matrices in $S^1_2$,
and its two dimensional subspace with the diagonal scalar  repeated, are not  unital operator spaces.
Note that $L^1_2$ and $S^1_2$, contain a two dimensional unital operator space,
namely the diagonal, which is a copy of $\ell^1_2$ (a unital operator space as we said above).  This gives a  glimpse 
of the delicacy of the process of adding an element or two to an operator space and trying to keep the space unital.

\begin{proposition} \label{2s} Let $X$ be an operator space with element $e \in {\rm Ball}(X)$.
Inside $M_2(X)$ consider the set ${\mathcal U}_e(X)$ of
matrices $\left[ \begin{array}{cl} \lambda e & x \\ 0 &  \lambda e\end{array} \right]$
 for $x \in X,$ and $\lambda$ scalar.
Then  $(X,e)$ is a unital operator space iff
$({\mathcal U}_e(X), e \otimes I_2)$ is a unital
 operator space, and iff
$({\mathcal U}_e(X), e \otimes I_2)$ is a unital operator algebra
with the canonical product. \end{proposition}

\begin{proof}  Suppose that $({\mathcal U}_e(X), e \otimes I_2)$ is a unital
operator space.  If $x  \in X, \Vert x \Vert = 1$, then
$$\sqrt{2} = \left| \left| \left[ \begin{array}{ccl} e & 0 \\
0 & e \\
0 & x \\
0 & 0 \end{array} \right] \right| \right| =
\left| \left| \left[ \begin{array}{cl} e \\ x \end{array} \right] \right| \right| .$$
Similarly for matrices, and so by the main theorem in \cite{BN},
$e$ is an isometry.  Similarly
$e$ is a coisometry, so $(X,e)$ is a unital operator space.  The rest is obvious.
\end{proof}

As pointed out in \cite{BN}, any theorem characterizing
unital operator spaces `linear-metrically', is also
a characterization of unitaries in $X$, that is
of elements of $X$ that are a unitary in some TRO containing 
$X$.   We give a
slight refinement of the main result in \cite{BN}, which
we will need later:

\begin{theorem} \label{uospr}  Let $X$ be 
an operator space, and fix $m,n \in \Ndb$.  An element $u \in M_{mn}(X)$ 
is a coisometry (resp.\ an isometry)
in $M_{mn}(X)$ (in the sense defined at the end of the introduction) 
iff $\Vert [ u_k 
\; \; x ] \Vert^2 = 1 + \Vert x \Vert^2$ (resp.\   $\Vert [ u_k \; \; x
]^t \Vert^2 = 1 + \Vert x \Vert^2$), for all $k
 \in \Ndb$ and $x \in M_{km}(X)$ (resp.\   $x \in M_{kn}(X)$).
 Indeed, it suffices to consider norm one matrices $x$ here.
 \end{theorem}

\begin{proof}   We just sketch this, since
it is similar to the proof of the main theorem in \cite{BN}, which 
the reader might follow along with.  It also uses 
facts about the ternary envelope of $M_{mn}(X)$ from 
e.g.\ \cite{Ham} or
\cite[Appendix A.13 (ii)]{BSh}, such as if  $Z$ is the ternary envelope of
$X$, then $M_{mn}(Z)$ is the
ternary envelope of $M_{mn}(X)$.
  We just prove the coisometry case, the other is
similar.  Let $c = (u u^*)^{\frac{1}{2}} \in M_m(Z^* Z)$.
Then $\Vert c x \Vert = \Vert u^* x \Vert = 1$ if 
$x \in M_m(X)$ with $\Vert x \Vert = 1$,   as in 
the proof we are following.
As in that proof, 
left multiplication by $c$ on $M_m(Z)$ is an isometry, since it 
restricts to an isometry on $M_m(X)$.  
By \cite[Theorem 2.1]{BN}, $c$ is a coisometry in 
$M_m(Z^* Z)$. Since $c \geq 0$ and $u u^* \geq 0$,
 by the unicity of positive square roots we must have $c = I$, 
and $u u^* = I$.
Thus $u$ is a coisometry. 
\end{proof}
 
We next  relate unital operator spaces to our previous paper
\cite{BNc}, by characterizing compact projections
in a $C^*$-algebra in terms of unital operator spaces.
We first mention some background facts from \cite{ORPI}.
Let $p$ be an open projection in the sense of Akemann \cite{Ake,Ake2}
or \cite{BHN},
 in the bidual of an approximately
 unital operator  algebra $A$, and let $q = 1-p$.  We recall that
$q$ is compact iff $q = aq$ for some $a \in {\rm Ball}(A)$.  We write
$A_p = \{ a \in A : a = ap \}$ and $_pA =  \{ a \in A : a = pa \}$.
It is easy to see that  the bidual of  $X = A/(_pA + A_p)$ is
 the unital operator  algebra  $qA^{**}q$.  Indeed
 consider the complete quotient map $x \mapsto q x q$
from $A^{**}$ onto $qA^{**}q$.  Its kernel is easily seen to be
$p A^{**} + A^{**} p$.  In particular the latter space
is weak* closed.   Thus
$A^{**}/(p A^{**} + A^{**} p) \cong qA^{**}q$ completely
isometrically.   Next, note that the weak* closure of
$_pA + A_p$ equals $p A^{**} + A^{**} p$ (using the
fact that  the latter space
is weak* closed).   Thus we have
$$(A/(_pA + A_p))^{**} \cong A^{**}/(p A^{**} + A^{**} p)
\cong qA^{**}q$$ completely isometrically.

 \begin{proposition}  Suppose that $B$ is a $C^*$-algebra, and that
$q = 1-p$ is a closed
projection in $B^{**}$.
  Then $q$ is compact if and only if $X = B/(_pB + B_p)$  is a
 unital operator space (i.e.\ iff it possesses a unitary in $X$
in the sense of the introduction).
\end{proposition} \begin{proof}  Let $i: X \to qB^{**}q$ be the canonical complete isometry induced by the canonical map from $X$ into its bidual,
and the identification in the last centered line above the proposition.  Explicitly,
$i([a]) = qaq$ for $a \in A$.

For one direction of the result, if $q$ is compact, so that
$q = aq$ for an $a \in {\rm Ball}(A)$,
let $e = [a] = a + (_pA + A_p) \in X$, and note that
 $i(e)$ is the identity $qaq = q$ of $qA^{**}q$.
So $X$ is a unital operator space.

For the other direction,
by \cite[Lemma
5.3]{BHN}, the TRO $Z$ generated by $i(X)$ inside $qB^{**}q$ is
a  ternary envelope of $X$, so that $i(v)$ is a coisometry in $Z$
by \cite[Lemma 2.3]{BN}, where
$v$ is the identity of $X$.
Thus $i(v)$ is a partial isometry in $qB^{**}q$.
Also,
$i(v) i(v)^* i(x)  = i(x)$ since $i(x) \in Z$ and $i(v)$ is a coisometry in
$Z$.
By weak* density of
$i(X)$ in $qB^{**}q$ we have $i(v) i(v)^* = q$.
Next note that by a result of Kirchberg (see the
remark after Corollary 1.3
of \cite{ORPI}), there exists $a \in {\rm Ball}(B)$
with $a + (_pB + B_p) = v$.   Then $qaq = i(v)$ so that $$i([a^* a])
 = q a^* a q = q a^* q a q  + q a^* (1-q) a q = q + q a^* (1-q) a q .$$
Taking norms, $1 \geq 1 + \Vert q a^* (1-q) a q \Vert$,
so that $q a^* (1-q) a q = 0$.
Hence $i([a^* a]) = q a^* a q = q$, from which it is clear that
 $q$ is compact.
\end{proof}

{\bf Remark.}  One may weaken the condition that $X$ is unital, to that
$X$ possesses an isometry or coisometry in the sense of the introduction.  The
proof above still works.  We also suspect that the result is also true
for general operator algebras (using the compact projections of \cite{BNc}).
 
\section{Characterizations of operator systems and $C^*$-algebras}

This section can be viewed as some remarks that 
naturally belong with the sections on operator systems in \cite{BN}.
We begin with a characterization of the 
`positive' part of a unital operator space:
 
\begin{lemma} \label{ref} If $A$ is
 a unital operator space or approximately unital
operator algebra,  then an element  $x \in {\rm Ball}(A)$ 
is in the positive cone of $A \cap A^*$  iff $\Vert 1 - zx \Vert 
\leq 1$  for all $z \in \Cdb$ with $|1-z| \leq 1$.
\end{lemma}

\begin{proof}   This follows from the argument for Lemma 8.5 in
\cite{BRead}.  \end{proof}  

The last result easily leads to a metric-linear characterization of
operator systems: they are the unital operator spaces $X$ spanned by
the positive cone of $A \cap A^*$, the latter characterized in 
Lemma \ref{ref}.    We now give
another metric-linear characterization of
operator systems which should have been stated in \cite{BN}.
We use the notation above Lemma \ref{sf}.
 
\begin{proposition}   \label{chos}  If $X$ is an operator space
possessing a 
conjugate linear involution $* : X \to X$, 
and an element $v \in {\rm Ball}(X)$ with $v = v^*$, then 
there exists a $*$-linear complete isometry $T : X \to Y$ onto
an operator system  $Y$ with $T(v) = 1$, iff 
$$\Vert r^{v_n}_x \Vert = \sqrt{1 + \Vert x \Vert^2} \; ,
\qquad n \in \Ndb, \; x \in M_n(X) .$$   
\end{proposition}  \begin{proof}  The one direction
follows  from (\ref{rx}).  For the other, first note that 
we have for any $x \in X$ that 
$\Vert x^* \Vert \leq \Vert r^{v}_x \Vert 
= \sqrt{1 + \Vert x \Vert^2}.$
Replacing $x$ by $tx$ for a positive scalar $t$ 
we obtain
$\Vert x^* \Vert \leq \sqrt{\frac{1}{t^2}  + \Vert x \Vert^2}$.
Letting $t \to \infty$ shows that $*$ is contractive, hence 
isometric since $*$ has period $2$.  
Similarly, $\Vert x^* \Vert = \Vert x \Vert$ 
if $x \in M_n(X)$, so that $X$ is a selfadjoint operator space 
by the discussion a
few paragraphs above Proposition 1.1 in \cite{BKNW}. 

 If $\Vert x \Vert = 1$,
and we 
 write 
the first row of $r^v_x$ as $a$, and the second of $r^v_x$ as $b$, then 
the norm of
each of these is $\leq \sqrt{2}$.  However $$4 \leq \Vert r_x \Vert^2 =
\Vert a^* a + b^* b \Vert \leq \Vert a \Vert^2 + \Vert b \Vert^2
\leq 4 .$$ Thus $\Vert a \Vert = \sqrt{2}$.  Similarly, 
the second column of $r^v_x$ has norm $\sqrt{2}$.  This works
analogously at the
matrix level, and we may now appeal to the main Theorem in \cite{BN}
to see that $(X,v)$ is a unital operator space.
Finally, we appeal to 3 (c) of \cite[Remark 3.5]{BN}  to see that
$(X,v)$ is an operator system.  \end{proof}

{\bf Remarks.}  \ 1) \ The following discussion rules out a possible simplification 
of the 
last characterization of operator systems.
In \cite{BN} we proved that if $X$ is a `selfadjoint
function space' with a selfadjoint `unitary' $u$,
then $(X,u)$ is a `function system' (the `commutative' variant
of an operator system).   The analoguous thing 
for operator systems is not true.
Indeed consider the selfadjoint operator space
 $X = \{ [x_{ij}] \in M_2 : x_{11} = 0, x_{12} = x_{21} \}$.
It is easy to see that $X$ generates $M_2$, so that $M_2$
is the ternary envelope of $X$.  However if
$u = E_{21} + E_{12}$ then $u X^* u = u X u$ is not contained in $X$.
It follows from the discussion at the start of Section 4 in
\cite{BN} that $(X,u)$ is not an operator system.  
This example also shows that
a selfadjoint operator space which is completely isometric
to a unital operator algebra, need not be completely isometric
to a $C^*$-algebra.

\medskip

2) \  One may ask if the equation (\ref{sx}) characterizes 
operator systems.  That is, 
if $u$ is a a selfadjoint unitary in a selfadjoint subspace
$X \subset B(H)$, then does the condition
$\Vert s^u_x \Vert =  1 +  ||x||$ for all $n \in \Ndb$ and $x \in M_n(X)$, force
$(X,u)$ to be an operator system?   Indeed a variant of the proof of Proposition \ref{chos} shows that
the equation $\Vert s^u_x \Vert =  1 +  ||x||$ above forces
$(X,u)$ to be a unital operator space.
 We leave this question to the interested
reader, suspecting that it is not hard to find a counterexample, and that 
 it is also not hard to find other simple conditions to
add to (\ref{sx}) to yield a characterization of operator systems.

\medskip

The following is a new metric-linear characterization of unital $C^{*}$-algebras 
among the operator systems, up to complete isometry.   The metric-linear characterizations
of $C^{*}$-algebras in  \cite{BN} referred to unitaries being spanning, which is
avoided here.  The result is certainly not
best possible, but the point again is that it is nice to know that
formulae exist that essentially {\em only} refer to the norm.

\begin{theorem}
\label{finno} An operator system $A$ has a product with respect to which it is
a $C^{*}$-algebra (with the same operator space structure)  
if and only if  for all  $x,y \in A$, there exist elements $b,z \in A$ such that
$$M_{+}=\frac{\left[ \begin{array}{ccccccl} y & 0&1&x&b&z \\ x & b&z&y&0&1
\end{array}  \right]}{\Vert\left[ \begin{array}{ccccccl} y & 0&1&x&b&z\\ x & b&z&y&0&1
\end{array}  \right] \Vert} \; \; , \; \; \; \; M_{-}=\frac{\left[ \begin{array}{ccccccl} y & 0&1&x&b&z \\
x & b&z&-y&0&-1
\end{array}  \right]}{\Vert\left[ \begin{array}{ccccccl} y & 0&1&x&b&z\\
x & b&z&-y&0&-1 \end{array}  \right] \Vert} $$
satisfy $||[M_+ \otimes I_{m},w]|| = \sqrt{2}$
and $||[M_- \otimes I_{m},w]|| = \sqrt{2}$,     
 for all $m\in \mathbb{N}$ and all contractions $w \in M_{2m}(A)$.  \end{theorem}

\begin{proof} Suppose the condition involving $M_{+}$ and $M_{-}$ 
holds for all $x,y \in A$ and let $S$ denote the $C^*$-envelope of $A$. Abusing notation, consider $A$
as canonically embedded in $S$. By Theorem \ref{uospr} the above
condition guarantees that $M_{+}$ and $M_{-}$ are coisometries in 
the ternary envelope
$M_{2,6}(S)$ of $M_{2,6}(A)$.  This implies that
$(xy^{*} + z) \pm (yx^{*} +z^{*}) = 0$ in $S$, so that
$z = -xy^{*}$ lies in $S$.  Hence $A$ is a subalgebra of its $C^*$-envelope, and
hence $A$ coincides with its $C^*$-envelope.

Conversely, suppose $A$ is linearly completely isometric to a $C^*$-algebra
$B$ via a map $\Psi : A \to B$.    Since $A$
is a unital operator space, it has a ternary envelope $(C,i)$
which is a unital $C^*$-algebra $C$ with $i(1) = 1_C$
(see e.g.\ \cite{Ham,BSh}).  By the theory of the
ternary envelope \cite{Ham},  we may take $(B,\Psi)$
to be a ternary envelope of $A$, and by the 
universal property of the ternary envelope there exists a
one-to-one ternary isomorphism $\theta : B \to C$ with
$\theta \circ \Psi = i$.   It follows that $i$ is
surjective, so that $A$ with its original
identity has a product with respect to which
it is a $C^*$-algebra.   In this $C^*$-algebra, let  $z = -x y^*$ and let
$b =  \sqrt{\Vert x x^{*}+ y y^{*}+z z^{*} \Vert \cdot 1-x x^{*}
-y y^{*} - z z^{*}}.$
It is now easy to check that $M_+$ and $M_-$ are coisometries.
 The result then follows from Theorem \ref{uospr}.  \end{proof}

{\bf Remark.}   An
operator system linearly completely isometric to a TRO or unital operator algebra
has a product with respect to which it is a $C^{*}$-algebra (with the same
operator system structure as the original one).   Indeed, the last theorem is true if we 
replace `$C^{*}$-algebra' in the statement with `TRO', or `unital operator algebra',
or `unital $C^{*}$-algebra' or `unital $C^{*}$-algebra with $1$ mapping to $1$'.
The proofs of these are usually the same as the proof of \ref{finno}, 
except in the `unital operator algebra' case 
where one should also  use \cite[Proposition 4.2]{BN}.  

 It is not true
however that an operator system which is linearly completely isometric to an
operator algebra, need be completely isometric to a $C^{*}$-algebra.  
Thus one cannot characterize $C^{*}$-algebras as
operator systems with a general operator algebra product.
For a counterexample, consider the operator system in \cite[Proposition
2.1]{BM1}, which by that result and Sakai's theorem 
 cannot be completely isometric to a $C^{*}$-algebra.
However the multiplication $(x,y) \mapsto x k y$ for a fixed contraction $k$
in the image of the compact operators in $X$, makes $X$ an operator algebra
by Remark 2 on p.\ 194 of \cite{BRS}.    
  
\section{Characterizations of operator algebra products}

In the last sections of our paper we will consider characterizations
of operator algebras.  The first point to be made
is that although we have not found one as yet,
there {\em ought} to be a purely linear-metric characterization of unital operator
algebras.  Indeed, we know from the noncommutative Banach-Stone theorem
that the identity in a unital operator algebra $A$ determines the
product (this is true even if the identity is one-sided \cite[Corollary 5.3]{Bone}).  Moreover, if we have forgotten the product on a unital operator algebra $A$
it can be recovered from the unital operator space structure by the methods of e.g.\
\cite[Section 6]{MCAA}.  These methods certainly 
yield a characterization of unital operator
algebras using only the `unital operator space data', but they are not quite 
`linear-metric' in our strict  
sense, since they refer to certain linear maps on $A$ for example.  
The second point is that it is still open
as to whether there is a truly metric condition 
on a bilinear map $m : X \times X \to X$ on an operator space $X$
 characterizing when $m$ is a (nonunital) operator algebra product.  
We offer in the remainder of the paper two partial contributions to 
these subjects.  In most of the present section we focus on this
second point in the case that $X$ posseses an isometry
or coisometry  (which need not be even a one-sided 
identity for the ensuing operator algebra
product).  Thus we are giving variants and extensions of 
the characterization of operator algebras from \cite{BRS}.  
In Section 6, we address the first point
with a linear-metric characterization of operator
algebras which does use elements of a containing $C^*$-algebra.
In \cite{NR2} a holomorphic characterization of operator algebras
is given (generalizing the  holomorphic characterization of $C^*$-algebras from
\cite{NR}).

\begin{lemma} \label{leco}
 If an operator algebra $A$ contains a left identity $u$ of norm $1$,
then $u$ is a coisometry in $A$ in the sense of the introduction.
\end{lemma}

\begin{proof}  This follows from e.g.\ \cite{Kext} or the considerations
involved in  \cite[Theorem 4.4]{Bone}, 
but we give a quick proof of it
using the main theorem in \cite{BN} (or Theorem \ref{uospr}  above),
as a nice application
of that result.  Note that $u$ is a projection in any containing
$C^*$-algebra, and so if $x = u x \in A$ has norm $1$ then
$$\Vert [ u \; \; x ] \Vert^2 = \Vert u u^* + x x^* \Vert = \Vert u + u x x^* u \Vert =
1 + \Vert ux \Vert^2 = 1 + \Vert x \Vert^2 = 2 .$$
Similarly for matrices, so that $u$ is a coisometry in $A$ by Theorem \ref{uospr}.
 \end{proof}

For some of the characterizations
of operator algebras below  we will use the  
quasimultiplier formulation of operator algebras \cite{KP}.  For the
readers convenience we  include a simple unpublished
proof of this that was in a preliminary version of \cite{BLM}, and 
which was presented e.g.\ at the 
Banach Algebras 2007 conference.  Here $I(X)$ is
the injective envelope of $X$, which is a TRO containing 
the ternary envelope of $X$ as a subTRO
 (see e.g.\ \cite{Ham} or \cite[Section 4]{BLM}).

\begin{theorem} \label{KP}  {\rm (Kaneda-Paulsen)} \
Let  $X$ be an operator space.  The
algebra products on $X$ for which
there exists a completely isometric homomorphism from $X$
onto an operator algebra, are in a bijective correspondence
with the
 elements $z \in {\rm Ball}(I(X))$
such that $X z^* X
\subset X$.  For such $z$ the associated
operator algebra product on $X$ is $x z^* y$.
\end{theorem}   \begin{proof}  The one direction, and the last
statement, follows
from Remark 2 on p.\ 194 of \cite{BRS}, viewing
$I(X)$ as a TRO in $B(H)$, and $V = z^*$.
For the other direction, if $X$ is a subalgebra of $B(H)$ say,
then by the theory of the injective envelope (see e.g.\
\cite{Ham} or \cite[Section 4]{BLM}) we can view
$X \subset I(X) \subset B(H)$, and there exists
a completely contractive projection $P$ from $B(H)$ onto
$I(X)$.  Set $z = P(1)$.   For $x,y \in X$
 we have $xy = P(x1^*y) = P(x P(1)^* y)$, by Youngson's theorem
 \cite[Theorem 4.4.9]{BLM}, the proof 
of which asserts that  the last quantity is the ternary product $x z^* y$ in $I(X) = {\rm Ran}(P)$.  The
bijectivity follows from e.g.\ \cite[Proposition
4.4.12]{BLM} and its `right-hand version': if $X z^* X = (0)$ then $X z^* = (0) = 
X z^* z$, so $z^* z = 0 = z$.  \end{proof}

If the operator space has more structure then
one can say more (see \cite{Kext}):  

\begin{corollary} \label{kpc}
Let  $(X,u)$ be a unital operator space.  The
algebra products on $X$ for which
there exists a completely isometric homomorphism from $X$
onto an operator algebra, are in a bijective correspondence
with the elements  $w \in {\rm Ball}(X)$ such that
$X w X \subset X$ (multiplication taken in the $C^*$-envelope
$C^*_{\rm e}((X,u))$).
For such $w$ the associated
operator algebra product on $X$ is $x w y$.
\end{corollary}   \begin{proof}  This
follows immediately from Theorem \ref{KP}, since in this setting
$I(X)$ may be taken to be a $C^*$-algebra, containing
$C^*_{\rm e}(X)$ as a $C^*$-subalgebra,
 with common identity $u$.  If $z$ is as in Theorem \ref{KP}
then $w = z^* = u z^* u \in X z^* X \subset X$.               
  \end{proof}

{\bf Remark.}  The elements $w$ in the unital operator space $X$ in the corollary
constitute the unit ball of the operator algebra
$D = \{ a \in C^*_{\rm e}(X) : X a X \subset X \}$,
which is a subalgebra of $C^*_{\rm e}(X)$.
Thus $D$ could justly be called {\em the operator algebra of
operator algebra products on}  $X$.  
It would be quite desirable to find
a linear-metric characterization of $D$ as a subset of $X$.      

\medskip

We recall the  spaces ${\mathcal M}_{\ell}(X)$ and ${\mathcal M}_r(X)$, 
of left and right  multipliers of $X$, which were
 introduced in \cite{BSh}.  Such multipliers of $X$ were 
`metric-linearly'
characterized by the first author, Effros and Zarikian
(see \cite[Theorem 4.5.2]{BLM}).  For example, if $T : X \to X$ is linear,
then  $T \in {\rm Ball}({\mathcal M}_{\ell}(X))$ iff
$$\left\| \left[ \begin{array}{cl} T(a_{ij}) \\
b_{ij} \end{array} \right] \right\| \leq \left\| \left[ \begin{array}{cl} a_{ij} \\
b_{ij} \end{array} \right] \right\| \, \quad
[a_{ij}], [b_{ij}] \in M_n(X) , n \in \Ndb .$$
  
\begin{lemma} \label{comu}  If $u$ is a coisometry 
in an operator space $X$, then the map $\theta : {\mathcal M}_{\ell}(X) \to X$
defined by $T \mapsto T(u)$, is
a complete isometry from ${\mathcal M}_{\ell}(X)$ into $X$.
Indeed if $T \in {\mathcal M}_{\ell}(X)$ then the ${\mathcal M}_{\ell}(X)$ norm of $T$
equals $\Vert T \Vert_{\rm cb} = \Vert T \Vert = \Vert T (u) \Vert$.
The range of $\theta$ is the set 
$X_\ell(u)$ defined to be $\{ x \in X : x u^* X \subset X , \, x = x u^* u \}$,
product taken in a ternary envelope $Z$ of $X$, and ${\mathcal M}_{\ell}(X)$ may be identified with
$\{ x u^* \in Z Z^* : x \in X_\ell(u) \}$.
If $u$ is a unitary and $Z$ is  
the $C^*$-envelope
$C^*_{\rm e}(X)$ of $(X,u)$, 
then ${\mathcal M}_{\ell}(X)$ is identified with
$\{ a \in X :  a X \subset X \}$ (product taken in $C^*_{\rm e}(X)$).
\end{lemma}

\begin{proof}   If $u$ is a coisometry, so that $u u^* = 1_{Z Z^*}$, then 
by the theory of one-sided multipliers of operator spaces we
may view ${\mathcal M}_{\ell}(X) \subset Z Z^*$
(see e.g.\ the 5th and
6th last lines
of p.\ 302 in \cite{BSh}).    Then $\theta$ is simply right multiplication 
by $u$.
It follows that $\theta$ is a 
complete isometry, since it is a complete contraction with
completely contractive left inverse $x \mapsto x u^*$.  Since it is well known that 
the ${\mathcal M}_{\ell}(X)$ norm of $T$
dominates $\Vert T \Vert_{\rm cb}$, the asserted norm equalities 
hold.    Clearly ${\mathcal M}_{\ell}(X) = \{ x u^* \in Z Z^* : x \in {\rm Ran}(\theta) \}$.
If $T \in {\mathcal M}_{\ell}(X)$ then $T u u^* X = T X \subset X$, so $\theta(T) \in X_\ell(u)$.
Conversely, if $x \in X_\ell(u)$, then $x u^* \in {\mathcal M}_{\ell}(X)$ so that $x = x u^* u \in {\rm Ran}(\theta)$.
Hence ${\rm Ran}(\theta) = X_\ell(u)$.  The rest is obvious.
\end{proof}
 
{\bf Remark.}  Write $Z_2(u)$ for the Pierce 2-space of $u$, this is 
a $C^*$-algebra in the natural `Pierce' product (see
e.g.\ p.\ 230-231 in \cite{BN0}).
The set $X_\ell(u)$ above equals the set 
of elements $z \in Z_2(u)$ such that 
left multiplication in the Pierce product by $z$ maps $X$ into $X$
(that is, $z u^* X \subset X$).
 It is easy to see that $X_\ell(u)$ is a unital subalgebra
of the $C^*$-algebra $Z_2(u)$ in the Pierce product.

\begin{theorem}  \label{chop}  Let $u$ be a coisometry 
in an operator space $X$.  Suppose that $m : X \times X \to X$ is a 
bilinear map such that $m(x,\cdot) \in {\mathcal M}_{\ell}(X)$
for all $x \in X$.  We also suppose that $m(\cdot,u) \in {\rm Ball}({\mathcal M}_r(X))$ (resp.\ $m(\cdot,u) \in  {\mathcal M}_r(X)$).
Then $m$ (resp.\ $m$ multiplied by some positive scalar) is an associative product
 such that $X$ with this product is
an operator algebra (that is, there
exists a completely isometric homomorphism 
from $X$ onto an operator algebra).
  Conversely,
every operator algebra product $m$ on $X$ 
satisfies all the conditions above.
\end{theorem} 

\begin{proof}  In the respectively case we can multiply
$m$  by a positive scalar to ensure that
$m(\cdot,u) \in {\rm Ball}({\mathcal M}_r(X))$.
 Let $Z$ be the ternary envelope of $X$, and view $X \subset Z$
and
${\mathcal M}_{\ell}(X) \subset Z Z^*$ as in the proof of Lemma \ref{comu}.
Define $v(x) = m(x,\cdot)  \in Z Z^*$, so that 
$m(x,y) = v(x) y$.  Similarly, we can view $m(\cdot,u)$
as a contractive right multiplier of 
$Z^* Z$, hence as a contraction $R$ in $(Z^* Z)^{**}$. 
Thus
 $$m(x,y) = v(x) y = v(x) u u^* y =
m(x,u) u^* y = x (Ru^*) y, \qquad x,y \in X .$$     Now the result follows from Remark 2 on p.\ 194 of \cite{BRS}.
\end{proof}

{\bf Remarks.}  1) \  In the previous theorem, the element $u$ need not be related to any identity, or 
 one-sided identity, for the ensuing operator algebra product.   

\smallskip

2) \ Theorem \ref{chop} answers the last question in \cite{KP} for  operator spaces containing a coisometry or isometry,
and in fact in this case gives a stronger result than the one discussed there.  

\smallskip

3) \ If $u$ is a unitary in $X$, then the `respectively'
assertion of Theorem \ref{chop} is true with the positive 
scalar mentioned there equal to $1$, if we also ask that $m$ be contractive as
a bilinear map.  This follows from a slight modification of the proof, 
using the other-handed version of Lemma \ref{comu} (indeed,   
 the multiplier norm of a right multiplier $T$ of a  unital operator space 
$(X,u)$ equals $\Vert T(u) \Vert$). 
  
\begin{corollary} \label{chco}  Let $u$ be a
coisometry in an operator space $X$ (or equivalently,
 $\Vert [ u_n \; \; x ] \Vert = \sqrt{2}$ for $n \in \Ndb$ and
every matrix $x \in M_n(X)$ of norm $1$).
Suppose that $m : X \times X \to X$ is a 
bilinear map such that $m(x,\cdot) \in {\mathcal M}_{\ell}(X)$
for all $x \in X$.  We also suppose that $m(x,u) = x$ for 
all $x \in X$. Then $m$ is an associative product,
and $X$ with this product is completely isometrically 
isomorphic to an operator algebra with
a two-sided identity (namely, $u$). 
Conversely, every unital operator algebra satisfies
all the conditions above.   
 \end{corollary}

\begin{proof}  By
the theorem,  $X$ with product $m$ is
an operator algebra.  Since $u$ is a right identity
it is an isometry by the `other-handed' variant of Lemma
\ref{leco}.    By \cite[Lemma 2.3]{BN}, $u$ is a unitary in $A$, 
hence is a unitary in the ternary envelope $Z$.
Then $Z$ is a $C^*$-algebra with product $x u^* y$.
In the proof of the last theorem, $R = 1$, and $m(x,y) = x u^* y$, 
and now it is clear that $u$ is a two-sided identity for $m$.      
\end{proof} 
 
Corollary \ref{chco} takes  longer to state than the 
characterization of unital operator algebras from \cite{BRS}.  However
the latter characterization is in terms of a product of {\em two large matrices},
whereas the condition 
that $m(x,\cdot) \in {\mathcal M}_{\ell}(X)$ in Corollary \ref{chco},
is as discussed 
above Lemma \ref{comu}, essentially the requirement that
$$\left\| \left[ \begin{array}{cl} m(x,a_{ij}) \\
b_{ij} \end{array} \right] \right\| \leq \left\| \left[ \begin{array}{cl} a_{ij} \\
b_{ij} \end{array} \right] \right\| \, \quad
[a_{ij}], [b_{ij}] \in M_n(X) , $$
for $n \in \Ndb$ and $x \in {\rm Ball}(X)$.    
We emphasize that this 
uses a small ($1 \times 1$) matrix (namely $x$) and one
large $n \times (2n)$ matrix, and in particular uses at most one operation
in each entry of the matrix, as opposed to the 
many operations (sums and products) that appear in the 
entries of a product of two large matrices.  

\section{Metric characterizations of operator algebras
referencing a containing $C^*$-algebra}

In 3(c) of \cite[Remark 3.5]{BN} the authors gave a metric-linear characterization, 
related to our norm formulae for 
$r_x$ above, of the adjoint
$x^{*}$  of any operator $x$ in an operator system $X$. 
(In particular, if $x$ and $z$ are contractive operators then
$z = x^{*}$ if and only if for all $t \in \mathbb{R}$ 
$$\left\| \left[ \begin{array}{ccl} t \cdot 1& x\\
-z&t \cdot 1
\end{array} \right] \right\| \le \sqrt{1+t^{2}}.)$$

Knowing this, we may freely reference adjoints in the following metric
characterization of operator algebras:

\begin{theorem} Suppose $A$ is a subspace of a unital $C^*$-algebra $B$. Then  $A$ is closed under multiplication if and only if for each pair of elements $x \in A$ and $y\in A^{*}$ with
$\Vert y \Vert \leq 1$, there exists an element $z \in A$ such that for all $b \in B$,
 $$\left\| \left[ \begin{array}{ccccl} 0 & y & 1 & 0 \\
2 \cdot 1 & x & z & b
\end{array} \right] \right\|  = ||[2 \cdot 1, x,z,b]|| . $$ 
\end{theorem}
\begin{proof} If $A$ is closed under multiplication, we may choose $z = -xy^{*}$. Multiplying the above matrix by its adjoint we see that the condition holds.  Conversely, suppose the condition holds.  Let $b = \sqrt{||xx^{*}+zz^{*}||\cdot 1-xx^{*}-zz^{*}}$. Multiplying the above matrix by its adjoint we see that 
$$\left\|  \left[ \begin{array}{ccl} yy^{*}+1&yx^{*}+z^{*}\\
xy^{*}+z&(4 + ||xx^{*}+zz^{*}||) \cdot 1
\end{array} \right] \right\|  = \max \{||[y,1]||,||[2 \cdot 1, x,z,b]||\}^2,$$
and this maximum actually equals $$\max
 \{1+||y||^{2},4 + ||xx^{*}+zz^{*}||\} = 4 + ||xx^{*}+zz^{*}||,$$ since
$\Vert y \Vert \leq 1$.  This implies that
$$\sqrt{||xy^{*}+z||^{2}+(4 + ||xx^{*}+zz^{*}||)^{2}} \le 4 + ||xx^{*}+zz^{*}||,$$ 
hence $||xy^{*}+z||=0$ and $xy^* = -z \in A$.   \end{proof}

  Clearly many other algebraic conditions $A$ might satisfy 
can be characterized by a variant of this theorem. 
For example,  a pair $x,y \in A$ satisfies $xy = 1$ if and only for all $b \in B$, the displayed condition is satisfied with $z=1$. Orthogonality, commutivity, normality and any other algebraic condition may be similarly characterized. We emphasize the following:

\begin{corollary}
Suppose that $A$ is a subspace of a unital $C^*$-algebra $B$. 
\begin{enumerate}
\item If  $x \in A$ then $x A \subset A$ 
 if and only if for all  $y\in {\rm Ball}(A^{*})$ there exists an element 
$z \in A$ such that for all $b \in B$,
 $$\left\| \left[ \begin{array}{ccccl} 0 & y & 1 & 0\\
2 \cdot 1 & x & z & b
\end{array} \right] \right\| = ||[2 \cdot 1,x,z,b]|| , $$
\item If $y \in {\rm Ball}(A)$ then $A y \subset A$ 
 if and only if for all  $x \in A$ there exists an element $z \in A$ 
such that for all $b \in B$,
 $$\left\|  \left[ \begin{array}{ccccl} 0 & y^* & 1 & 0 \\ 2 \cdot 1 &  x & z & b
 \end{array} \right]
\right\| = ||[2 \cdot 1, x,z,b]|| , $$
\item If $x \in A$ then $A x A \subset A$ if and only if for all  $y\in {\rm Ball}(A^{*})$
 there exists an element $z \in \{ b \in B : A b \subset A \}$ such that for all $b \in B$,
the equality in {\rm (1)} holds.
\end{enumerate}
\end{corollary}

Note that $A$ is a unital operator space and $B$ its $C^*$-envelope,
for example, then the last result characterizes left, right, and 
quasi- multipliers of $A$ (see the last assertion of
Lemma \ref{comu} and Corollary \ref{kpc}).

We thank Matt Gibson for helpful discussions relating to Section 6.

\end{document}